\documentclass[10pt,a4paper]{amsart}
\usepackage[utf8]{inputenc}
\usepackage[T1]{fontenc}
\usepackage[french, english]{babel}
\usepackage{amsmath,amsfonts,amssymb,amsthm}
\usepackage{stmaryrd} 
\usepackage{array}
\usepackage{enumerate}
\usepackage{hyperref}
\usepackage{mathabx}
\usepackage{comment}
\usepackage{ifthen}
\usepackage{lineno}

\newcommand{\N}{\mathbb{N}}
\newcommand{\Z}{\mathbb{Z}}
\newcommand{\Q}{\mathbb{Q}}
\newcommand{\R}{\mathbb{R}}
\newcommand{\C}{\mathbb{C}}
\newcommand{\Gal}{\mathrm{Gal}}
\newcommand{\tors}{\mathrm{tors}}
\newcommand{\poubelle}[1]{}

\theoremstyle{plain}
\newtheorem{thm}{Theorem}[section]
\newtheorem{prop}[thm]{Proposition}
\newtheorem{cor}[thm]{Corollary}
\newtheorem{lmm}[thm]{Lemma}

\newtheorem{conj}[thm]{Conjecture}
\newtheorem{rqu}[thm]{Remark}

\newtheorem*{thm*}{Theorem}
\newtheorem*{lmm*}{Lemma}

\begin{document}
\poubelle{\linenumbers}

\title[Bogomolov property of nonabelian extensions]{Bogomolov property of some infinite nonabelian extensions of a totally $v$-adic field}
\date{}
\author{Arnaud Plessis}
\address{Arnaud Plessis: Academy of Mathematics ans Systems Science, Morningside Center of Mathematics, Chinese Academy of Sciences, Beijing 100190}
\email{plessis@amss.ac.cn}
\maketitle 

\begin{abstract}
Let $E$ be an elliptic curve defined over a number field $K$, and let $v$ be a finite place of $K$. 
Write $K^{tv}$ for the maximal totally $v$-adic field, and denote by $L$ the field generated over $K^{tv}$ by all torsion points of $E$. 
Under some conditions, we will show that the absolute logarithmic Weil height (resp. Néron-Tate height) of any element of $L$ (resp. $E(L)$) is either $0$ or bounded from below by a positive constant depending only on $E,K$ and $v$. 
This constant will be explicit in the toric case.
\end{abstract}

\section{Introduction}
Let $h : \overline{\Q} \to \R$ denote the (absolute, logarithmic) Weil height. 
It is a non-negative function vanishing precisely at $\mu_\infty$, the set of all roots of unity, and $0$ by a theorem of Kronecker. 
It satisfies $h(\alpha^n)= \vert n \vert h(\alpha)$ and $h(\zeta\alpha)=h(\alpha)$ for all $\alpha\in \overline{\Q}, \zeta\in \mu_\infty$ and all $n\in \Z$  as well as the inequality $h(\alpha\beta)\leq h(\alpha)+h(\beta)$ for all $\alpha,\beta\in \overline{\Q}$. 
For further information on this height, we refer to \cite{BombieriandGubler}.

Given a field $K\subset \overline{\Q}$, an interesting question is whether there exists a positive constant $c$ such that $h(\alpha) \geq c$ for all non-zero $\alpha\in K\backslash \mu_\infty$.
Such a field is said to have the Bogomolov property.
This notion was introduced  by Bombieri and Zannier \cite{BombieriZannier}. 
The field $\overline{\Q}$ does not have the Bogomolov property since $h(2^{1/n})=(\log 2)/n \to 0$. 

By Northcott's theorem, each number field has the Bogomolov property. 
Schinzel gave the first example of an infinite extension of $\Q$ having the Bogomolov property \cite{Schinzel}, namely the maximal totally real field extension $\Q^{tr}$ of $\Q$.  
The $p$-adic version of this theorem was proved by Bombieri and Zannier \cite{BombieriZannier}. 
More precisely, they proved that the maximal totally $p$-adic extension $\Q^{tp}$ of $\Q$ has the Bogomolov property. 

In recent years, the study of this property mushroomed, see for example \cite{AmorosoDvornicich, AmorosoZannier10, Habegger, AmorosoDavidZannier, Grizzard, FiliMilner, Galateau, Frey, Plessismino, PazukiPengo}.

The study of this property is not limited to this situation and we can easily define it for abelian varieties. 
Let $A$ be an abelian variety defined over a number field $K$, and let $\mathcal{L}$ be a symmetric ample line bundle on $A/K$.
Let $\hat{h}_A : A(\overline{K}) \to \R$ denote the Néron-Tate height attached to $\mathcal{L}$. 
It is a non-negative function vanishing precisely at $A_\tors$, the group of torsion points of $A$. 
Again, given a field $L\subset \overline{K}$, the group $A(L)$ is said to have the Bogomolov property (with respect to $\mathcal{L}$) if there exists a positive constant $c$ such that $\hat{h}_A(P) \geq c$ for all $P\in A(L)\backslash A_\tors$. 
It is well known that $A(\overline{K})$ does not have the Bogomolov property. 

Northcott's theorem cited above also states that $A(L)$ has the Bogomolov property if $L$ is a number field. 
Zhang showed the abelian analogue of Schinzel's theorem, that is, $A(\Q^{tr})$ has the Bogomolov property \cite{Zhang}.  
Later, Baker and Petsche proved that $A(\Q^{tp})$ has the Bogomolov property when $p>2$ and $A/\Q$ is an elliptic curve with semistable reduction at $p$ \cite[Theorem 6.6]{BakerPetsche}.
For more examples concerning the Bogomolov property in the case of an abelian variety, see \cite{BakerSilverman, PazukiPengo} (which handle the case of any abelian variety) and \cite{Baker, Silverman, Habegger, Pottmeyer, Plessisiso} (which treat the special case of an elliptic curve). 

A very special case of a recent conjecture due to the author predicts the following. 

\begin{conj} [\cite{Plessisiso}, Conjecture 1.4] \label{conj0}
Let $A$ be an abelian variety defined over a number field $K$, let $\mathcal{L}$ be a symmetric ample line bundle on $A/K$, and let $L/K$ be a finite extension. 
Then $L(A_\tors)$ and $A(L(A_\tors))$ have the Bogomolov property.  
\end{conj}

\begin{rqu}
\rm{The abelian part of this conjecture is due to David.}
\end{rqu}

This conjecture was proved to be true when $A$ has complex multiplication (CM). 
More precisely, the toric part is due to Amoroso, David and Zannier (see Theorem \ref{thm5} below for a more general statement) and the abelian part was proved by Baker and Silverman \cite[Section 9]{BakerSilverman}, see also \cite[Théorème 1.8]{Carrizosa}. 

The case where $A$ has no CM is much harder. 
Up to my knowledge, Habegger was the first one to provide a result going in the direction of Conjecture \ref{conj0}.

\begin{thm} [Habegger, \cite{Habegger}] \label{thm-1}
Let $E/\Q$ be an elliptic curve. 
Then $\Q(E_\tors)$ and $E(\Q(E_\tors))$ have the Bogomolov property. 
\end{thm}

For any elliptic curve $E$ and any integer $N\in\N=\{1,2,\dots\}$, write $E[N]$ for the group of $N$-torsion points of $E$ and define $j_E$ to be the $j$-invariant of $E$. 

Set $\mathrm{Mat}_2(A)$ the ring of square matrices with $2$ lines whose coefficients lie in a ring $A$. 
The set of its invertible elements is denoted with $\mathrm{GL}_2(A)$.
Define $\mathrm{SL}_2(A)$ as the kernel of the determinant map $\mathrm{GL}_2(A) \to A^*$ (here, $A^*$ is the set of invertible elements in $A$). 

Given a number field $K$, a finite place $v$ of $K$ and an algebraic extension $L/K$, we say that $L$ has bounded local degree at $v$ if $d_v(L)=\sup_w [L_w : K_v]$ is finite, where $w$ ranges over all extensions of $v$ to $L$. 
In such a case, we denote by $e_w(L\vert K)$, resp. $f_w(L\vert K)$, the ramification index, resp. inertia degree, of the extension $w\vert v$.
Finally, we define $K^{tv}$ as the maximal totally $v$-adic field, that is, the set of $\alpha\in \overline{K}$ such that $v$ is totally split in $K(\alpha)$. 
It is Galois over $K$ and $d_v(K^{tv})=1$. 

Recently, Frey pointed out a quite remarkable fact: Conjecture \ref{conj0} may be true for some infinite extensions $L/K$. 

\begin{thm} [Frey, \cite{Freypubli}, Theorem 7.1] \label{thm1}
Let $E/\Q$ be a non-CM elliptic curve, and let $L/\Q$ be a Galois extension such that the exponent $\exp(L)$ of its Galois group is finite.
Then there exists a rational prime $p$ satisfying:
\begin{itemize} 
\item [(a)] $E$ has supersingular reduction at $p$ and $j_E \not\equiv 0, 1728 \; (p)$; 
\item [(b)] the natural representation $\Gal(\Q(E[p])/\Q) \to \mathrm{GL}_2(\Z/p\Z)$ is surjective; 
\item [(c)] $p \geq \max\{2\sup_q \{d_q(L)\}+2; \exp(L)\}$, where $q$ runs over all rational primes, 
\end{itemize}
and for such a $p$, we have, for all $\alpha\in L(E_\tors)^* \backslash \mu_\infty$, 
\[h(\alpha) \geq \frac{(\log p)^4}{p^{5p^4}}.\]
\end{thm}

\begin{rqu}
\rm{By a theorem of Checcoli \cite[Theorem 1]{Checcoli}, if $L/\Q$ is Galois, then the exponent of its Galois group is finite if and only if $\sup_q\{d_q(L)\}$ is finite, where $q$ ranges over all rational primes.
So item $(c)$ makes sense here.}
\end{rqu}
 
The main goal of this paper is to establish that Conjecture \ref{conj0} is true for some Galois extensions $L/K$ whose Galois group has infinite exponent.

\begin{thm} \label{thm2}
Let $E$ be an elliptic curve defined over a number field $K$, and let $L/K$ be a finite Galois extension.
If there is a finite place $v$ of $K$ satisfying
\begin{itemize} 
\item [(a)] $E$ has supersingular reduction at $v$ and $j_E\not\equiv 0, 1728 \; (v)$;
\item [(b)] the image of the natural representation $\Gal(L(E[p])/L) \to \mathrm{GL}_2(\Z/p\Z)$ contains $\mathrm{SL}_2(\Z/p\Z)$, where $p\Z=v \cap \Z$; 
\item [(c)] $p> \max\{3, 2 d_v(L)\}$;
\item [(d)] $e_v(K\vert \Q)=1$ and $f_v(K \vert \Q) \leq 2$,
\end{itemize}
then for all $\alpha\in LK^{tv}(E_\tors)^* \backslash \mu_\infty$, we have
\[h(\alpha) \geq \frac{1}{4^{p^2 d_v(L)}+1}\left(\frac{\log p}{d_v(L)(40\sqrt{2}+2)[K:\Q] p^{2p^2d_v(L)+2}}\right)^{ 2+ \frac{4}{p^{p^2d_v(L)/4}-2}}.\]
Moreover, if $v$ is unramified in $L$ and if the natural representation in (b) is surjective, then $E(LK^{tv}(E_\tors))$ has the Bogomolov property. 
\end{thm}

\begin{rqu}
\rm{$a)$ Lemma \ref{lmm5.10} prevents us from providing an explicit lower bound of the N\'eron-Tate height for points lying in $E(LK^{tv}(E_\tors))$.

$b)$ 
Assume that $K=\Q(\sqrt{D})$ with $D\in\N$ and that $E/K$ has no CM.
Then $(a)$ is satisfied for infinitely many places by Elkies' thesis \cite{Elkies}. 
The natural representation in $(b)$ is surjective for all but finitely many rational primes by Serre's open image theorem \cite{Serrejp}. 
Item $(c)$ holds for all $p$ large enough since $d_v(L)\leq [L : K]$. 
Finally, all but finitely many finite places of $K$ are unramified in $L$ and satisfy $(d)$. 
So we can find a place $v$ of $K$ satisfying all conditions of Theorem \ref{thm2}. 
Thus $LK^{tv}(E_\tors)$ and $E(LK^{tv}(E_\tors))$ have the Bogomolov property. 
In particular, Conjecture \ref{conj0} is true for elliptic curves defined over a real quadratic field.

Nonetheless, Theorem \ref{thm2} does not permit us to treat the case $D<0$ in full generality.
For example, we do not know so far if the elliptic curve  \[E : iy^2 = x^3+(i-2)x^2+x\]
defined over $\Q(i)$ has at least one place of supersingular reduction (it is however conjectured that there exist infinitely many) \cite[Section 5.2]{Elkies}. 

$c)$ Our lower bound is much stronger than that of Theorem \ref{thm1}.
Let us see this through a concrete example.
Consider the elliptic curve
\[E : y^2+y=x^3-x^2-10x-20\]
defined over $K=\Q$.
According to \cite[elliptic curve 11.a2]{LMFDB}, $E$ has conductor $N=11$ and $j$-invariant $j_E=-2^{12} \; 11^{-5} \; 31^3$.
By the same reference, $(b)$ with $L=\Q$ holds for all $p \geq 7$.
Next, $j_E\not\equiv 0 \; (p)$ for all $p\notin \{2;11;31\}$ and $j_E\not\equiv 1728 \; (p)$, that is, 
\[ 2^6 \; 41^2 \; 61^2 = 11^5 \; 1728+2^{12} \; 31^3 \not\equiv 0 \; (p), \]
for all $p \notin \{2; 41; 61\}$. 
Finally, $E$ has supersingular reduction at $p=19$ \cite[Chapter 5, Example 4.6]{Silvermanarithmetic}. 
From all this, Theorem \ref{thm2} claims that for all $\alpha\in \Q^{t19}(E_\tors)^* \backslash \mu_\infty$, 
 \[h(\alpha) \geq  \frac{1}{4^{19^2}+1}\left(\frac{\log 19}{(40\sqrt{2}+2)19^{2 \cdot 19^2+2}}\right)^{2+ \frac{4}{19^{19^2/4}-2}} \geq 2.6 \cdot 10^{-2072}.\]
We cannot deduce this lower bound from Theorem \ref{thm1} because $\Gal(\Q^{t19}/\Q)$ has infinite exponent. 
Consider a number field $F\subset \Q^{t19}$ of degree $d\geq 9$.
Even under this restriction, it is not always possible to get the lower bound above from Theorem \ref{thm1} since $p=19$ is not a suitable choice here, item $(c)$ being not satisfied.
 
Let $n\in \N$ be an integer and write $\mathcal{V}(n)=\sum_{p \; \text{prime} \; \leq n} \log p$. 
We have the inequality $\mathcal{V}(n) < 1.01624n$ \cite{RosserSchoenfeld}.
Applying \cite[Theorem 4.13]{Frey} to $M=2d+2$ proves the existence of a rational prime $p$ between $n=\max\{2d+2, 7654\}$ and 
\[e^{1.3 \; 10^8 \; e^{2\mathcal{V}(n)+\frac{11}{15}e^{\mathcal{V}(n)}}} < e^{e^{e^{\mathcal{V}(n)}}} < e^{e^{e^{1,01624 n}}}\]
such that $E$ has supersingular reduction at $p$.
For such a choice of $p$, items $(a)-(c)$ of Theorem \ref{thm1} are all three satisfied, which leads to the lower bound 
\[\forall \alpha\in F(E_\tors)^* \backslash \mu_\infty, \; h(\alpha) \geq \left(e^{e^{e^{e^{1.02n}}}}\right)^{-1}.\]
We can compare the two lower bounds above and check that ours is much better. } 
\end{rqu}

Our theorem suggests that Conjecture \ref{conj0} can be extended as follows. 

\begin{conj} \label{conj1}
Let $A$ be an abelian variety defined over a number field $K$. 
Let $\mathcal{L}$ be a symmetric ample line bundle on $A/K$, and let $L/K$ be an algebraic extension. 
If $d_v(L)$ is finite for at least one finite place $v$ of $K$, then $L(A_\tors)$ and $A(L(A_\tors))$ have the Bogomolov property.  
\end{conj}

The best argument in favor of this statement is probably the result below. 

\begin{thm} [Amoroso-David-Zannier] \label{thm5}
Let $A$ be a CM abelian variety defined over a number field $K$. 
Let $\mathcal{L}$ be a symmetric ample line bundle on $A/K$, and let $L/K$ be a Galois extension. 
If $d_v(L)$ is finite for at least one finite place $v$ of $K$, then $L(A_\tors)$ has the Bogomolov property. 
\end{thm}

\begin{proof}
As $A$ is CM, there exits a finite Galois extension $M/K$ such that $M(A_\tors)/M$ is abelian. 
Choose $\sigma\in \Gal(LM(A_\tors)/LM)$ and $\tau\in \Gal(LM(A_\tors)/M)$.
If $\alpha\in M(A_\tors)$, then $\sigma\tau\alpha = \tau\sigma \alpha$ since $M(A_\tors)/M$ is abelian. 
If $\alpha\in LM$, then $\sigma(\tau \alpha) = \tau(\sigma\alpha) = \tau\alpha$ since $LM/M$ is Galois (because $L/K$ is by assumption) and $\sigma$ fixes the elements of $LM$. 
From all this, we get $\sigma\tau=\tau\sigma$, i.e., $\Gal(LM(A_\tors)/LM)$ is contained in the center of $\Gal(LM(A_\tors)/M)$. 
As $d_v(LM)$ is bounded from above by $d_v(L)d_v(M)< +\infty$, the theorem now arises from \cite[Theorem 1.2]{AmorosoDavidZannier}.
\end{proof}

\subsection*{Acknowledgement} 
I thank P. Habegger and L. Pottmeyer for replying to my questions as well as F. Amoroso and L. Terracini for pointing out a mistake in an earlier version of this text. 
This work was funded by Morningside Center of Mathematics, CAS. 

\section{An elementary result}
Write $\langle X\rangle$ for the group generated by a subset $X$ of a group $G$. 
Let $L/K$ be a Galois extension of number fields, and let $w$ be a finite place of $L$.
Set $D(w\vert w\cap K)$ the decomposition group of the extension $w\vert w\cap K$, that is, the set of $\psi\in\Gal(L/K)$ such that $\psi w = w$. 

Fix for this section a number field $K$ as well as a finite place $v$ of $K$. 
For any finite extension $L/K$, we write $V_L$ for the set of places of $L$ above $v$.

\begin{lmm} \label{lmm3.1}
Consider a totally $v$-adic finite Galois extension $M/K$ and a tower of number fields $K\subset K'\subset L$ with $L/K'$ Galois.
Assume that \[H':=\left\langle \bigcup_{w\in V_L} D(w\vert w\cap K')\right\rangle=\Gal(L/K').\] 
Then \[H:=\left\langle \bigcup_{w\in V_{LM}} D(w\vert w\cap K'M)\right\rangle=\Gal(LM/K'M).\] 
\end{lmm}   

\begin{proof}
Let $w\in V_{LM}$, and let $\mathrm{Res} : \Gal(LM/K'M) \to \Gal(L/K')$ be the restriction map. 
It is injective and induces a homomorphism from $\Gal((LM)_w/ (K'M)_w)$ to $\Gal(L_w/K'_w)$, and so from $D(w\vert w\cap K'M)$ to $D(w\cap L\vert w\cap K')$. 
As $M$ is a totally $v$-adic field, we have $M_w=K_v$; whence $\Gal((LM)_w/(K'M)_w)=\Gal(L_w/K'_w)$.
In particular, $D(w\vert w\cap K'M)$ and $D(w\cap L\vert w\cap K')$ have the same cardinality, and so $\mathrm{Res} : D(w\vert w\cap K'M) \to D(w\cap L\vert w\cap K')$ is an isomorphism for all $w\in V_{LM}$. 
Hence, $\mathrm{Res} : H \to H'$ is an isomorphism too. 
By assumption, we have the chain of inclusions $\Gal(L/K')=H'=\mathrm{Res} (H) \subset \mathrm{Res}(\Gal(LM/K'M))\subset \Gal(L/K')$ and the lemma follows. 
\end{proof}

Keep the notation of Lemma \ref{lmm3.1} and assume that both $K'/K$ and $L/K$ are Galois. 
Let $w$ be a finite place of $L$. 
Then $\psi D(w\vert w \cap K') \psi^{-1} = D(\psi w\vert \psi w \cap K')$ for all $\psi\in\Gal(L/K)$.
The fact that $\Gal(L/K)$ acts transitively on $V_{L}$ leads to 
\begin{equation} \label{eq0}
\langle \psi D(w\vert w \cap K') \psi^{-1}, \psi\in\Gal(L/K)\rangle = \left\langle \bigcup_{w'\in V_{L}} D(w'\vert w'\cap K')\right\rangle.
\end{equation}

\begin{cor} \label{cor1} 
Consider a totally $v$-adic finite Galois extension $M/K$ and a tower of number fields $K\subset K'\subset L$ with $K'/K$ and $L/K$ Galois.
Let $w$ be a place of $LM$ above $v$ and assume that \[\Gal(L/K')=\langle \psi D(w\cap L \vert w\cap K')\psi^{-1}, \psi\in \Gal(L/K)\rangle.\] 
 If $\gamma\in LM$ with $\sigma \gamma\in K'_w$ for all $\sigma\in \Gal(LM/K)$, then $\gamma\in K'M$. 
\end{cor}

\begin{proof}
By assumption, it arises from \eqref{eq0} that \[\Gal(L/K')= \left\langle \bigcup_{w'\in V_{L}} D(w'\vert w'\cap K')\right\rangle.\] 
Using Lemma \ref{lmm3.1}, then \eqref{eq0} applied to $L=LM$ and $K'=K'M$, gives
\begin{align*}
\Gal(LM/K'M) & = \left\langle \bigcup_{w'\in V_{LM}} D(w'\vert w'\cap K'M)\right\rangle \\ 
& = \langle \psi D(w\vert w\cap K'M)\psi^{-1}, \psi \in\Gal(LM/K) \rangle.
\end{align*}
We have $M_w=K_v$ since $M$ is a totally $v$-adic field. 
Thus $D(w\vert w\cap K'M)$ is equal to $\Gal((LM)_w/(K'M)_w)=\Gal(L_w/K'_w)$. 
The lemma follows since $\gamma$ is fixed by $\psi D(w\vert w\cap K'M)\psi^{-1}$ for all $\psi \in\Gal(LM/K)$. 
\end{proof}

\section{Some results extracted from \cite{Freypubli}. } \label{Section 2}
For any number field $K$ and any finite place $v$ of $K$, we denote by $K_v$ the completion of $K$ with respect to $\vert .\vert_v$, the normalized $v$-adic absolute value, that is, $\vert p\vert_v=p^{-1}$, where $p\Z= v \cap \Z$. 
Then, write $K_v^{ur}$ for the maximal unramified extension of $K_v$ and $\Q_{p^2}$ for the unramified extension of degree $2$ of $\Q_p$ inside $\overline{\Q_p}$.

In \cite[Section 3]{Freypubli}, Frey fixed the following notation: a non-CM elliptic curve $E/\Q$, a Galois extension $L/\Q$ whose Galois group has finite exponent, a rational prime $p$ satisfying the conditions $(a)-(c)$ of Theorem \ref{thm1}, a number field $K\subset L$, which is Galois over $\Q$, and a finite Galois extension $F/\Q_{p^2}$ containing $K_v$, where $v$ denotes the place of $K$ associated to a fixed field embedding $\overline{\Q}\to\overline{\Q_p}$. 

Actually, we can prove most of results mentioned in \cite{Freypubli} without involving most of conditions above.
Strictly speaking, we should reprove them all using only the minimal conditions. 
But they are very technical, making it impossible without considerably burden this text. 
As a compromise, we mention below the hypotheses/references that Frey used to prove each one of her results, then we detail one by one the requested conditions to use these references.  

\subsection{Results extracted from \cite[Section 3]{Freypubli}}
Here, $p$ denotes a rational prime.
In \cite[Lemma 3.1]{Freypubli}, she used $[F:\Q_p]<p, \Q_{p^2} \subset F$ and \cite[Lemma 3.4]{Habegger}. 
In \cite[Lemma 3.2]{Freypubli}, she used $[F:\Q_p]<p, \Q_{p^2} \subset F$ and \cite[Lemma 3.3, Lemma 3.4]{Habegger}.
In \cite[Lemma 3.3]{Freypubli}, she used $[F:\Q_p]<p, \Q_{p^2} \subset F$ and \cite[Lemma 3.3]{Habegger}. 
In \cite[Lemma 3.4]{Freypubli}, she used the fact that $F/\Q_{p^2}$ is a finite Galois extension as well as \cite[Lemma 2.1, Lemma 3.3, Lemma 3.4]{Habegger} and \cite[Proposition II.7.12]{Neukirch}.
In \cite[Lemma 3.5]{Freypubli}, she used the fact that $F/\Q_{p^2}$ is a finite Galois extension as well as \cite[Lemma 2.1, Lemma 3.3]{Habegger} and \cite[Lemme IV.5, Proposition IV.12]{Corpslocaux}.
In \cite[Lemma 3.6]{Freypubli}, she used \cite[Lemma 3.5]{Habegger}, \cite[Proposition II.7.13]{Neukirch}, Goursat's lemma, $[F:\Q_p]<p$ and $\Q_{p^2} \subset F$.
In \cite[Lemma 3.7]{Freypubli}, she used \cite[Proposition II.7.12]{Neukirch}.
In \cite[Lemma 3.8]{Freypubli}, she used results of \cite{Freypubli}. 

Thus, to prove \cite[Lemma 3.1-Lemma 3.8]{Freypubli}, we only need to assume that $F/\Q_{p^2}$ is a finite Galois extension such that $[F:\Q_p]<p$ as well as the necessary conditions so that \cite[Lemma 2.1, Lemma 3.3-Lemma 3.5]{Habegger}, \cite[Proposition II.7.12-Proposition II.7.13]{Neukirch}, \cite[Lemme IV.5, Proposition IV.12]{Corpslocaux} and Goursat's lemma hold. 
Goursat's lemma is a general fact about group theory, \cite[Lemme IV.5, Proposition IV.12]{Corpslocaux} are general results about ramification groups and \cite[Proposition II.7.12- Proposition II.7.13]{Neukirch} are general facts about cyclotomic fields. 
Then \cite[Lemma 2.1]{Habegger} is a general lemma on local fields.  
Finally, the results lying in \cite[Section 3- 5]{Habegger} hold for every rational prime $p\geq 5$ and every elliptic curve $E$ defined over $\Q_{p^2}$ with supersingular reduction and whose $j$-invariant is neither $0$ nor $1728$ in the residual field of $\Q_{p^2}$.

In conclusion, all results of \cite[Section 3]{Freypubli} work in the following situation, that we will refer from now as $(S)$: 
\begin{itemize}
\item $p\geq 5$ is a rational prime; 
\item $E$ is an elliptic curve defined over $\Q_{p^2}$ with supersingular reduction and its $j$-invariant is neither $0$ nor $1728$ in the residual field of $\Q_{p^2}$;
\item $F/\Q_{p^2}$ is a finite Galois extension such that $[F:\Q_p] <p$.
\end{itemize}

Say $N\in\N$.
Denote by $\mu_N$ the set of all $N$-th roots of unity and by $\mathrm{Aut} E[N]$ the set of automorphisms of $E[N]$.
Let $L/K$ be a finite Galois extension of local fields, and let $\pi$ be the prime ideal of $L$. 
For $i\geq 0$, we define $G_i(L/K)$ as the $i$-th ramification group of $L/K$, that is, the set of $\psi\in \Gal(L/K)$ such that $\psi x - x \in \pi^{i+1}$ for all $x\in L$ with $\vert x\vert_\pi\leq 1$.
It is well known that $G_0(L/K)= \Gal(L/L\cap K^{ur})$. 

We can now state some results extracted from \cite[Section 3]{Freypubli}.

\begin{lmm} \label{lmm2.1}
Let $p,E$ and $F$ be as in $(S)$. 
Let $N\in\N$ be an integer with $p$-adic valuation $n$.
Then:
\begin{enumerate} [(i)]
\item The extension $F(E[p^n])/F(E[p])$ is totally ramified of degree $p^{2(n-1)}$; 
\item The extension $F(E[N])/F(E[p^n])$ is unramified; 
\item $\Gal(F(E[N])/F(E[N/p]))\simeq \Gal(F(E[p^n])/F(E[p^{n-1}]))\simeq (\Z/p\Z)^2$ if $n\geq 2$;
\item For $m\in\N$ coprime to $p$, the image of the representation $\Gal(F(E[p^n])/F)\to \mathrm{Aut} E[p^n]$ contains the multiplication-by-$m^{[F: \Q_{p^2}]}$ map.
\item For $M\in\N$ coprime to $p$, the order of $\Gal(F(E[pM])/F(E[M]))$ divides $p^2-1$;
\item $\Gal(F(E[N])/F(E[N/p]))\subset G_s(F(E[N])/F)$ if $n\geq 2$, where $s=p^{2(n-1)}-1$.
\item If $n\geq 2$, then $F(E[N]) \cap \bigcup_{m\in\N} \mu_{p^m} = \mu_{p^n}$. 
\item Say $n\geq 2$. 
If $\psi\in\Gal(F(E[N])/F(E[N/p]))$ and if $a\in F(E[N])^*$ satisfy $(\psi a/a)^{p^2}\neq 1$, then $\psi a/a\notin \mu_\infty$ (in particular, $(\psi a/a)^{p^2}\notin \mu_\infty$).
\end{enumerate}
\end{lmm}

\begin{proof}
See \cite[Lemma 3.3-3.6, Lemma 3.8]{Freypubli}. 
\end{proof}

\subsection{Some results extracted from \cite[Section 4]{Freypubli}}
In \cite[Lemma 4.1 (i),(ii), (iv),(v)]{Freypubli}, Frey only used \cite[Lemma 2.1, Lemma 3.1, Lemma 5.1]{Habegger}.
In \cite[Lemma 4.2-4.5]{Freypubli}, she only used results of \cite[Section 4]{Freypubli} which are not \cite[Lemma 4.1 (iii)]{Freypubli}. 
All these statements therefore hold in the situation $(S)$. 

\begin{lmm} \label{lmm2.2}
Let $p,E$ and $F$ be as in $(S)$ and put $\mathcal{E}=(p^2-1)[F: \Q_{p^2}]$.  
Take an integer $N\in\N$ not divisible by $p^2$ and denote by $n$ its $p$-adic valuation.
Then there is $\phi\in\Gal(F(E[N])/F(E[p^n]))$ such that: 
\begin{enumerate}
\item [(i)] $\phi$ acts on $E[N/p^n]$ as multiplication by $p^\mathcal{E}$; 
\item [(ii)] For all $a\in F(E[N])$, we have $\vert \phi a-a^{p^{2\mathcal{E}}}\vert_p \leq p^{-1/\mathcal{E}}\max\{1, \vert a\vert_p\}^{1+p^{2\mathcal{E}}}$. 
\end{enumerate}
\end{lmm}
 
\begin{proof}
 $(i)$ Let $\tilde{\phi}\in\Gal(\Q_p^{ur}/\Q_{p^2})$ be the lift of the Frobenius squared.
Write $N/p^n= \prod_l l^{v_l}$ for the decomposition of $N/p^n$ into a product of rational primes.
Let $l$ be a rational prime dividing $N/p^n$. 
Then $l\neq p$ and \cite[Lemma 3.2]{Habegger} implies that $\tilde{\phi}$ acts on $E[l^{v_l}]$ as multiplication by $\pm p$. 
The isomorphism  $\bigoplus_l E[l^{v_l}] \simeq E[N/p^n]$ being compatible with the action of the Galois group, we deduce that $\tilde{\phi}$ acts on $E[N/p^n]$ as multiplication by $\pm p$.
By \cite[Lemma 4.1 (ii)]{Freypubli}, there is $\phi\in\Gal(F(E[N])/F(E[p^n]))$ such that $\phi$ and $\tilde{\phi}^\mathcal{E}$ coincide on $E[N/p^n]$. 
This shows $(i)$ since $\mathcal{E}$ is even.

$(ii)$ It arises from \cite[Lemma 4.4]{Freypubli} and from the equality $\vert \phi(a)\vert_p=\vert a\vert_p$, which holds since two Galois conjugates of $\overline{\Q_p}$ have the same $p$-adic absolute value \cite[Chapter II, §2, Corollaire 3]{Corpslocaux}.
\end{proof}

A proof of the next lemma can be found in \cite[Lemma 3.5]{FreyActa}.
It is only based on elementary calculations.

\begin{lmm} \label{lmm2.3}
Let $0<\delta<1/2$, and let $\beta\in\overline{\Q}^*\backslash \mu_\infty$ be such that $[\Q(\beta) : \Q]\geq 16$ and $h(\beta)\leq 1/4$. 
Then \[ \frac{1}{[\Q(\beta):\Q]}\sum_{\tau: \Q(\beta)\hookrightarrow \C} \log\vert \tau\beta -1\vert \leq \frac{40}{\delta^4}h(\beta)^{\frac{1}{2}-\delta}. \]
\end{lmm}

\subsection{Some results extracted from \cite[Section 5-6]{Freypubli}}
In \cite[Lemma 5.1-5.2]{Freypubli}, Frey only used results proved in \cite[Section 3]{Freypubli}. 
So we can use \cite[Lemma 5.2]{Freypubli} under the conditions more general of the situation $(S)$, which gives:
\begin{lmm} \label{lmm2.4}
Let $p,E$ and $F$ be as in $(S)$, and let $N\in\N$ be an integer divisible by $p^2$. Then for all $a\in F(E[N])$ and all $\psi\in \Gal(F(E[N])/F(E[N/p]))$, we have \[ \vert \psi a^{p^2}-a^{p^2}\vert_p \leq p^{-1/[F:\Q_{p^2}]} \max\{1, \vert a\vert_p\}^{2p^2}.\]  
\end{lmm}
(Again, we exploited the fact that $\vert \psi a\vert_p=\vert a\vert_p$). 

The cardinality of a finite set $X$ is denoted with $\# X$.
Apparently, \cite[Lemma 5.3]{Freypubli} seems to involve the conditions $(a)-(c)$ of Theorem \ref{thm1}.
Actually, they are useless and we prove a more general fact below. 

\begin{lmm} \label{lmm2.5}
Let $L/K$ be a Galois extension of number fields, and let $H$ be a normal subgroup of $\Gal(L/K)$. 
Let $\psi\in H$, and set \[C = \{\sigma\in\Gal(L/K), \sigma\psi\sigma^{-1}=\psi\} \] the centralizer of $\psi$. 
Then for all finite places $w$ of $L$, the cardinality of the orbit $C w =\{\sigma w, \sigma\in C\}$ is at least  $[L:K]/([L_w:K_w]\#H)$. 
\end{lmm}

\begin{proof}
The orbit of $\psi$ under the conjugation action of $\Gal(L/K)$ on itself is included in $H$ since this latter is normal in $\Gal(L/K)$. 
The orbit-stabilizer theorem ensures us that $\#C\geq [L:K]/\#H$.  
Let $w$ be a finite place of $L$. 
The Galois group $\Gal(L/K)$ acts transitively on all places of $L$ above $w\cap K$ and the total number of such places is $[L:K]/[L_w:K_w]$.
So the orbit $C w$ has cardinality at least \[\frac{1}{[\Gal(L/K): C]}\frac{[L:K]}{[L_w:K_w]}\geq \frac{[L:K]}{[L_w:K_w]\#H},\] which concludes the proof of the lemma.
\end{proof}

The proof of \cite[Lemma 6.1]{Freypubli} only requires results present in \cite[Section 3]{Freypubli}. 

\begin{lmm} \label{lmm2.6}
Let $p,E$ and $F$ be as in $(S)$.
Let $N\in\N$ be an integer whose $p$-adic valuation $n$ is at least $2$. 
Take an integer $m\in\N$ coprime to $p$.
Then there is $\tau_m\in \Gal(F(E[N])/F)$ such that 
\begin{enumerate}
\item [(i)] $\tau_m$ acts as raising to the power of $m^{2[F : \Q_{p^2}] (p^2-1)}$ on $\mu_{p^n}$; 
\item [(ii)] $\tau_m$ acts as multiplication by $m^{[F:\Q_{p^2}](p^2-1)}$ on $E[p^n]$;
\item [(iii)] $\tau_m$ acts trivially on $E[N/p^n]$. 
\end{enumerate}
\end{lmm}

\begin{proof}
For $(i)$ and $(ii)$, see \cite[Lemma 6.1]{Freypubli} (Frey gave the proof for $m=2$, but it easily extends to $m$ coprime to $p$ thanks to Lemma \ref{lmm2.1} $(iv)$). 
For $(iii)$, see the last paragraph in the proof of \cite[Lemma 6.1]{Freypubli}.  
\end{proof}

The next statement is a general lemma in linear algebra.

\begin{lmm} \label{lmm2.7}
Consider an odd rational prime $p$.
Let $U$ be a $\Z/p\Z$-vector subspace of $\mathrm{Mat}_2(\Z/p\Z)$ with order greater than $p$ that contains at least one non-zero scalar matrix. 
Then $\langle AUA^{-1}, A\in \mathrm{SL}_2(\Z/p\Z)\rangle = \mathrm{Mat}_2(\Z/p\Z)$. 
\end{lmm}

\begin{proof}
See \cite[Lemma 6.4]{Freypubli}. 
\end{proof}

For the convenience of the reader, we give a (quick) proof of the last lemma of this section although it is only a "copy-paste" of that of \cite[Lemma 6.5 (i)]{Freypubli}.

\begin{lmm} \label{lmm2.8}
Let $p,E$ and $F$ be as in $(S)$, and let $L\subset F$ be a number field. 
Assume that $E$ is defined over $L$ and that the image of the natural representation $\Gal(L(E[p])/L) \to \mathrm{GL}_2(\Z/p\Z)$ contains $\mathrm{SL}_2(\Z/p\Z)$. 
Take $N\in\N$ such that its $p$-adic valuation $n$ is at least $2$ and put $G=\Gal(F(E[N])/F(E[N/p]))$.
Then \[H:=\left\langle \psi G \psi^{-1}, \psi\in \Gal(L(E[N])/L)\right\rangle = \Gal(L(E[N])/L(E[N/p]))=:H'. \]
\end{lmm}

\begin{proof}
Let $\rho : \Gal(L(E[N])/L) \to \mathrm{GL}_2(\Z/N\Z)$ be the natural representation.
As $n\geq 2$, it is well known that we can define an injective homomorphism $\mathcal{L}$ from $H'$ to $\mathrm{Mat}_2(\Z/p\Z)$ as follows: If $\sigma\in H'$, then $\mathcal{L}(\sigma)$ is the unique element of $\mathrm{Mat}_2(\Z/p\Z)$ satisfying $\rho(\sigma)= 1+(N/p)\mathcal{L}(\sigma)$, where $1$ denotes the identity matrix. 

By definition, $H$ is the normal closure of $G$ in $\Gal(L(E[N])/L)$; whence $H\subset H'$.
Let $\pi : \mathrm{GL}_2(\Z/N\Z)\to \mathrm{GL}_2(\Z/p\Z)$ be the natural projection. 
If $\psi\in G\subset H\subset H'$ and if $\sigma\in\Gal(L(E[N])/L)$, then $\sigma\psi\sigma^{-1}\in H\subset H'$ and an easy calculation gives
\begin{align*}
\rho(\sigma \psi \sigma^{-1}) &= \rho(\sigma)\rho(\psi)\rho(\sigma)^{-1} = \rho(\sigma)(1+(N/p)\mathcal{L}(\psi))\rho(\sigma)^{-1} \\ 
& = 1+(N/p)\rho(\sigma)\mathcal{L}(\psi)\rho(\sigma)^{-1} = 1+(N/p)\pi\rho(\sigma)\mathcal{L}(\psi)\pi\rho(\sigma)^{-1},
\end{align*} 
leading to $\mathcal{L}(\sigma\psi\sigma^{-1})=\pi\rho(\sigma)\mathcal{L}(\psi)\pi\rho(\sigma)^{-1}\in \mathcal{L}(H)$.
By assumption, the image of $\pi\rho$ contains $\mathrm{SL}_2(\Z/p\Z)$. 
It implies that $\langle A\mathcal{L}(G)A^{-1}, A\in \mathrm{SL}_2(\Z/p\Z)\rangle \subset \mathcal{L}(H)$. 

Lemma \ref{lmm2.1} $(iii)$ tells us that $G$, and so $\mathcal{L}(G)$, has cardinality $p^2$.
If $\mathcal{L}(G)$ contains a non-zero scalar matrix, then Lemma \ref{lmm2.7} applied to $U=\mathcal{L}(G)$ would prove that $\mathrm{Mat}_2(\Z/p\Z)= \mathcal{L}(H)\subset \mathcal{L}(H')\subset \mathrm{Mat}_2(\Z/p\Z)$, and so $H=H'$ by injectivity of $\mathcal{L}$.

Let $m$ be a generator of $(\Z/p^n\Z)^*$. 
As $[F:\Q_p]<p$, and so is coprime to $p$, it follows that $M=m^{[F:\Q_{p^2}](p^2-1)p^{n-2}}$ has order $p$ in $(\Z/p^n\Z)^*$. 
Moreover, $M\equiv 1 \; (p^{n-1})$ by Euler's theorem. 
Consequently, the multiplication-by-$M$ map has order $p$ in $\mathrm{Aut} E[p^n]$ and acts trivially on $E[p^{n-1}]$. 

Lemma \ref{lmm2.6} $(ii)$ applied to $N=p^n$ and $m=m^{p^{n-2}}$ tells us that there exists $\tau\in \Gal(F(E[p^n])/F)$ acting on $E[p^n]$ as multiplication by $M$. 
By the foregoing, $\tau$ is an element of $\Gal(F(E[p^n])/F(E[p^{n-1}]))$ with order $p$. 

By Lemma \ref{lmm2.1} $(iii)$, the restriction map $G\to \Gal(F(E[p^n])/F(E[p^{n-1}]))$ is an isomorphism. 
Let $\tilde{\tau}\in G$ be the element that gets mapped to $\tau$ under this map. 
As $\tilde{\tau}$ acts on $E[p^n]$ as scalar multiplication, we deduce that $\mathcal{L}(\tilde{\tau})\in\mathcal{L}(G)$ is a scalar matrix, which cannot be zero since $\tilde{\tau}$ has order $p$ and $\mathcal{L}$ is injective. 
\end{proof}

\section{Proof of Theorem  \ref{thm1}: toric case}

Fix for this section the notation (and assumptions) of Theorem \ref{thm2} in the toric case and a field embedding $\overline{K} \to \overline{K_v}$.
As everything is now fixed, we ease notation by putting $M(N)=M(E[N])$ for any field $M\subset \overline{K_v}$ and any integer $N\in\N$.

Item $(d)$ leads to either $K_v=\Q_p$ or $K_v = \Q_{p^2}$. 
Our elliptic curve is therefore defined over $\Q_{p^2}$. 
Moreover, by $(a)$, it has supersingular reduction and its $j$-invariant is neither $0$ nor $1728$ in the residual field of $\Q_{p^2}$.
Then, put $F=L_{w_0}\Q_{p^2}$, where $w_0$ is the place of $L$ associated to the fixed embedding $\overline{K}\to\overline{K_v}$.
It is Galois over $\Q_{p^2}$ since $L/K$ is Galois. 
Next, it follows from $(c)$ that $p\geq 5$ and 
\begin{equation} \label{eq40}
p> 2 d_v(L) \geq [\Q_{p^2} : \Q_p] [L_{w_0}: K_v]\geq [F:\Q_p]. 
\end{equation}
To summarize, our scope is a particular case of the situation $(S)$. 
By $(b)$, $E$ is defined over $L\subset F$ and the natural representation $\Gal(L(p)/L)\to \mathrm{GL}_2(\Z/p\Z)$ contains $\mathrm{SL}_2(\Z/p\Z)$. 
We thus have access to all the results of Section \ref{Section 2}.

The next two results will serve us both in the toric case and in the elliptic case.
Start by putting in place our descent argument.

\begin{lmm} \label{lmm4.0}
Let $N\in\N$ be an integer divisible by $p(p^2-1)$ such that \[\left\langle \psi \Gal(F(N)/F(N/p)) \psi^{-1}, \psi\in \Gal(L(N)/L)\right\rangle = \Gal(L(N)/L(N/p)). \]
Let $M/K$ be a totally $v$-adic finite Galois extension. 
If $\gamma\in LM(N)$ with $\sigma \gamma\in F(N/p)$ for all $\sigma\in \Gal(LM(N)/K)$, then $\gamma\in LM(N/p)$. 
\end{lmm}

\begin{proof}
Since $N/p$ is divisible by $p^2-1$, basic properties of the Weil pairing prove that $\zeta_{p^2-1}\in K(N/p)$. 
As $\Q_{p^2}=\Q_p(\zeta_{p^2-1})$, we get $\Q_{p^2}\subset K_v(N/p)\subset K_v(N)$. 

Denote by $w$ the place of $LM(N)$ associated to the fixed embedding $\overline{K}\to \overline{K_v}$.
Then $L(N)_w= L_{w_0}\Q_{p^2}(N)=F(N)$. 
Similarly, $L(N/p)_w= F(N/p)$. 
In conclusion, $\Gal(F(N)/F(N/p))= D(w\vert w\cap L(N/p))$.
The lemma now arises from Corollary \ref{cor1} applied to $K'=L(N/p)$ and $L=L(N)$.
\end{proof}

\begin{lmm} \label{lmm4.2}
Take an integer $N\in \N$ of $p$-adic valuation $n$ and $\psi\in\Gal(F(N)/F)$.
If $\psi$ acts as scalar multiplication on both $E[p^n]$ and $E[N/p^n]$, then it belongs to the center of $\Gal(LK^{tv}(N)/K)$.
In particular, the elements $\phi$ and $\tau_m$ introduced in Lemma \ref{lmm2.2} and Lemma \ref{lmm2.6}, respectively, lie in the center of $\Gal(LK^{tv}(N)/K)$.
\end{lmm}

\begin{proof}
Clearly, $\psi$ fixes $LK^{tv}\subset F$.
Taking the sum of points gives an isomorphism between $E[p^n] \times E[N/p^n]$ and $E[N]$, which is compatible with the action of $\Gal(\overline{K}/K)$.
We infer that $\psi$ must lie in the center of $\Gal(LK^{tv}(N)/K)$. 
\end{proof}

A proof of the well-known result below can be found in \cite[Lemma 2 (i)]{Dobrowolski}. 

\begin{lmm} \label{lmm4.01}
Let $a\in\overline{\Q}^*$, and let $\psi\in\Gal(\overline{\Q}/\Q)$. 
If $a\notin\mu_\infty$, then $\psi a^b/a^c\notin\mu_\infty$ for all integers $b,c\in\N$ distinct.
\end{lmm}

Let $\alpha\in LK^{tv}(E_\tors)^*\backslash \mu_\infty$. 
There is a totally $v$-adic finite Galois extension $M/K$ such that $\alpha\in LM(E_\tors)$. 
For brevity, put $L'=LM$.

The proof of the proposition below is largely inspired by that of \cite[Lemma 4.6]{Freypubli}.

\begin{prop} \label{prop1}
Let $N\in\N$ be an integer with $p$-adic valuation $n$, and let $a\in L'(N)^*\backslash \mu_\infty$.
Assume that $n\leq 1$ or that $n\geq 2$ and $a^{p^2}\notin F(N/p)$.  
Then \[ h(a)\geq k= \left(\frac{\log p}{d_v(L)(40\sqrt{2}+2)[K:\Q] p^{2p^2d_v(L)+2}}\right)^{2+\frac{4}{p^{p^2d_v(L)/4}-2}}. \] 
\end{prop}

\begin{proof}
Construct $\psi\in \Gal(L'(N)/L')$ as follows: If $n\leq 1$, then $\psi$ is the homomorphism $\phi$ of Lemma \ref{lmm2.2}.
Otherwise, $\psi$ is any element of $\Gal(F(N)/F(N/p))$ satisfying $\psi a^{p^2} \neq a^{p^2}$ (such an element exists by assumption). 
Next, put \[ t = \begin{cases}
0 \; \text{if} \; n\leq 1 \\ 
4 \; \text{if} \; n\geq 2
\end{cases}, \; \mathcal{E}=\begin{cases}
(p^2-1)[F:\Q_{p^2}] \; \text{if} \; n\leq 1 \\ 
[F: \Q_{p^2}] \; \; \; \; \; \; \; \; \; \; \; \; \; \text{if} \; n\geq 2
\end{cases}, \;  (b,c) = \begin{cases}
(1, p^{2\mathcal{E}}) \; \text{if} \; n\leq 1 \\ 
(p^2,p^2) \; \text{if} \; n\geq 2
\end{cases} \] and $x=\psi a^b - a^c$.
Note that the latter is non-zero (for $n\geq 2$, it is by construction and for $n\leq 1$, it is by Lemma \ref{lmm4.01}).

Denote by $v_0$ the place of $L'(N)$ associated to the fixed embedding $\overline{K}\to \overline{K_v}$. 
Let $C$ be the centralizer of $\psi$ in $\Gal(L'(N)/K)$. 
Lemma \ref{lmm4.2} gives $C= \Gal(L'(N)/K)$ if $n\leq 1$, and so the orbit $C v_0$ is the set of all places of $L'(N)$ above $v$. 
In particular, it has cardinality $[L'(N) : K]/[L'(N)_{v_0} : K_v]$.  
If $n\geq 2$, then Lemma \ref{lmm2.5} applied to $L=L'(N)$ and $H=\Gal(L'(N)/L'(N/p))$, which has cardinality at most $p^4$, proves that $C v_0$ has cardinality at least $[L'(N) : K]/(p^4[L'(N)_{v_0} : K_v])$.
To summarize, $Cv_0$ has cardinality at least $[L'(N) : K]/(p^t[L'(N)_{v_0} : K_v])$.

Let $w$ be a place of $L'(N)$. 
If $w$ is a finite place, the ultrametric inequality gives 
\begin{equation} \label{eq11}
\vert x\vert_w\leq \max\{\vert \psi a ^b\vert_w, \vert a\vert_w^c\}\leq \max\{1, \vert \psi a\vert_w\}^b \max\{1, \vert a\vert_w\}^c.
\end{equation}
If we further assume that $w\in C v_0$, then there is $\sigma\in C$ such that $w=\sigma^{-1}v_0$. 
Thus, \[\vert x\vert_w = \vert x \vert_{\sigma^{-1} v_0} = \vert \sigma x\vert_{v_0}= \vert \sigma(\psi a)^b-\sigma a^c\vert_{v_0}= \vert \psi(\sigma a)^b - \sigma a^c\vert_{v_0}.\]
Lemma \ref{lmm2.2} (if $n\leq 1$) or Lemma \ref{lmm2.4} (if $n\geq 2$) applied to $a=\sigma a$ gives 
\begin{equation} \label{eq12}
 \vert x\vert_w\leq p^{-1/\mathcal{E}} \max\{1, \vert \sigma a\vert_{v_0}\}^{b+c}= p^{-1/\mathcal{E}} \max\{1, \vert a\vert_w\}^{b+c}.
\end{equation}
If $w$ is an infinite place, we have to take a little detour. 
Put $\beta= \psi a^b/a^c\neq 1$ and note that $h(\beta)\leq h(\psi a^b)+h(a^c)= (b+c)h(\alpha)\leq 2p^{2\mathcal{E}} h(a)$. 
Moreover, $\beta\notin \mu_\infty$ (it is clear by Lemma \ref{lmm2.1} $(viii)$ if $n\geq 2$ and by Lemma \ref{lmm4.01} otherwise).
Clearly,
\begin{equation} \label{eq13}
\vert x\vert_w = \vert \beta -1\vert_w \vert a\vert_w^c\leq \vert \beta-1\vert_w \max\{1, \vert a \vert_w^{b+c}\}.
\end{equation}
Recall that $x\neq 0$.
Collecting \eqref{eq11}-\eqref{eq13}, it follows from the product formula that 
\begin{equation} \label{eq1}
\begin{aligned}
0 & = \sum_w  [L'(N)_w : \Q_p]\log\vert x\vert_w \\ 
& \leq \sum_{w\in C v_0} [L'(N)_w : \Q_p] \log\left(p^{-1/\mathcal{E}}\max\{1,\vert a\vert_w\}^{b+c}\right) \\ 
& + \sum_{w\notin C v_0, w\nmid \infty} [L'(N)_w : \Q_w]\log\left( \max\{1, \vert \psi a\vert_w\}^b\max\{1, \vert a\vert_w\}^c\right) \\ 
& + \sum_{w\vert \infty} [L'(N)_w : \Q_w]\log\left(\vert \beta-1\vert_w \max\{1, \vert a \vert_w^{b+c}\}\right). 
\end{aligned}
\end{equation}
As $L'(N)/K$ is Galois, the degree of the extension $L'(N)_w/K_v$ does not depend on the place $w$ of $L'(N)$ above $v$. 
Thus \[\sum_{w\in C v_0} [L'(N)_w : \Q_p]=[K_v : \Q_p] [L'(N)_{w_0} : K_v] \#(C v_0) \geq \frac{[K_v : \Q_p] [L'(N): K]}{p^t}.\]
After dividing \eqref{eq1} by $[L'(N) : \Q]$, we infer, thanks to a small calculation, that 
\begin{equation} \label{eq2}
 \frac{[K_v : \Q_p] \log p}{\mathcal{E} [K: \Q] p^t} \leq (b+c)h(a) + \frac{1}{[L'(N): \Q]} \sum_{w\vert \infty} [L'(N)_w : \Q_w] \log \vert \beta-1\vert_w.
\end{equation}
If $h(\beta)\geq 1/4$, then the proposition follows from the inequality $h(\beta)\leq 2p^{2\mathcal{E}} h(a)$. 
If $[\Q(\beta) : \Q]\leq 15$, then Dobrowolski's inequality \cite{Dobrowolski} gives \[ h(\beta) \geq \frac{1}{15} \log\left(1+\frac{1}{1200}\left(\frac{\log\log 15}{\log 15}\right)^3\right) \geq 10^{-6}\] and the proposition arises from the inequality $h(\beta)\leq 2 p^{2\mathcal{E}} h(a)$. 
If $h(\beta)\leq 1/4$ and $[\Q(\beta) : \Q]\geq 16$, then Lemma \ref{lmm2.3} applied to $\delta=1/p^{\mathcal{E}/4}$ gives \[ \frac{1}{[L'(N): \Q]} \sum_{w\vert \infty} [L'(N)_w : \Q_w] \log \vert \beta-1\vert_w \leq 40p^{\mathcal{E}}h(\beta)^{(1/2)-\delta}\leq 40\sqrt{2}p^{2\mathcal{E}} h(a)^{(1/2)-\delta}.  \]
The proposition is trivial if $h(a)\geq 1$. 
Otherwise, $h(a)< 1$ and from \eqref{eq2}, we get \[\frac{\log p}{\mathcal{E}[K:\Q]p^t} \leq 2p^{2\mathcal{E}}h(a)+ 40\sqrt{2}p^{2\mathcal{E}} h(a)^{(1/2)-\delta} \leq (40\sqrt{2}+2) p^{2\mathcal{E}} h(a)^{1/2-\delta} \] since $b+c \leq 2p^\mathcal{E}$.
Recall that $[F:\Q_p]= 2 [ F : \Q_{p^2}] \leq 2d_v(L)$ by \eqref{eq40}. 
So we have $\mathcal{E}\leq p^2d_v(L)$ and $2\mathcal{E}+t\leq 2p^2d_v(L)$.
We finally get \[h(a)\geq \left(\frac{\log p}{d_v(L)(40\sqrt{2}+2)[K:\Q] p^{2p^2d_v(L)+2}}\right)^{\frac{2}{1-2\delta}}. \] 
The proposition follows since $2/(1-2\delta) = 2+4/(p^{\mathcal{E}/4}-2) > 2+4/(p^{p^2 d_v(L)/4}-2)$. 
\end{proof}

\textit{Proof of Theorem \ref{thm2}, toric case:} 
Let $N\in\N$ be an integer such that $\alpha\in L'(N)$.
By enlarging $N$ if needed, we can assume that it is divisible by $p^2(p^2-1)$.
Let $n\geq 2$ denote the $p$-adic valuation of $N$.
Recall that $[F : \Q_p]<2d_v(L)$ by \eqref{eq40}. 

Put $\tau=\tau_2\in \Gal(F(N)/F)$ the homomorphism introduced in Lemma \ref{lmm2.6} as well as $\gamma= (\tau\alpha)/\alpha^D\in L'(N)$, where $D=4^{[F:\Q_{p^2}](p^2-1)}\leq 4^{p^2 d_v(L)}$.
We get \[h(\gamma)\leq h(\tau \alpha)+h(\alpha^D)= (1+D)h(\alpha)\leq (1+4^{p^2 d_v(L)}) h(\alpha)\] and $\gamma\notin\mu_\infty$ by Lemma \ref{lmm4.01}.
Our theorem would follow if we get $h(\gamma) \geq k$ (see Proposition \ref{prop1}).
Let $n'\in\N$ be the least integer such that $\sigma \gamma\in F(p^{n'-n}N)$ for all $\sigma\in \Gal(L'(N)/K)$. 
We have $n'\leq n$ since $\gamma\in L'(N)$. 

Show by decreasing induction on $t$ that $\gamma\in L'(p^{t-n} N)$ for all $t\in\{n', \dots, n\}$. 
The base case $t=n$ is obvious. 
We now assume that our assertion is true for $t>n'\geq 1$ and show that it also holds for $t-1$. 
Recall that $p^n$ divides $N$.

Clearly, $p^2(p^2-1)$ divides $N_t=p^{t-n}N$. 
Lemma \ref{lmm2.8} applied to $N=N_t$ gives \[\left\langle \psi \Gal(F(N_t)/F(N_t/p)) \psi^{-1}, \psi\in \Gal(L(N_t)/L)\right\rangle = \Gal(L(N_t)/L(N_t/p)).\]
By assumption, $\sigma \gamma \in F(N_t/p)$ for all $\sigma\in \Gal(L'(N_t)/K)$ and Lemma \ref{lmm4.0} applied to $N=N_t$ ends the induction. 
In particular, $\gamma\in L'(N')$ where $N'=N_{n'}$.

\underline{Case $n'= 1$}. 
As $\gamma\in L'(N')$ is neither $0$ nor a root of unity, we can apply Proposition \ref{prop1} to $N=N'$ and $a=\gamma$, which gives us $h(\gamma)\geq k$.

\underline{Case $n'\geq 2$}.
The minimality of $n'$ proves that there is $\sigma\in \Gal(L'(N)/K)$ such that $\sigma \gamma\notin F(N'/p)$. 
We want to apply Proposition \ref{prop1} to $N=N'$ and $a=\sigma \gamma$, which would prove our theorem since $h(\gamma)=h(\sigma\gamma)$. 
As $\gamma\in L'(N')$, it remains to show that $\sigma\gamma^{p^2} \notin F(N'/p)$. 
For this, assume by contradiction that it is the case.

Since $\sigma\gamma \notin F(N'/p)$, there is $\psi\in\Gal(F(N)/F(N'/p))$ such that $\psi\sigma\gamma \neq \sigma\gamma$. 
Moreover, $\psi\sigma\gamma^{p^2}= \sigma\gamma^{p^2}$ by assumption. 
Thus $\psi\sigma\gamma=\zeta\sigma\gamma$ for some $\zeta\in\mu_{p^2}\backslash \{1\}$. 
As $\tau$ commutes with both $\psi$ and $\sigma$ by Lemma \ref{lmm4.2}, we get
\begin{align*}
\zeta & = \frac{\psi\sigma\gamma}{\sigma\gamma}  = \frac{\psi((\sigma\tau\alpha)/\sigma \alpha^D)}{(\sigma\tau\alpha)/\sigma\alpha^D} = \frac{\tau(\psi\sigma\alpha)}{\tau(\sigma\alpha)} \frac{(\sigma \alpha)^D}{(\psi\sigma \alpha)^D} = \frac{\tau\eta}{\eta^D},
\end{align*}
where $\eta=(\psi \sigma \alpha)/\sigma \alpha$. 
As $\zeta\in\mu_\infty$, we have $\eta\in\mu_\infty$ by the contrapositive of Lemma \ref{lmm4.01}.
Let $T\in\N$ be an integer coprime to $p$ such that $\eta^T$ has order a power of $p$. 
Lemma \ref{lmm2.1} $(vii)$ gives $\eta^T\in\mu_{p^n}$ and Lemma \ref{lmm2.6} proves that $\tau\eta^T=(\eta^T)^D$. 
We conclude $\zeta^{p^2}=\zeta^T=1$, and so $\zeta=1$ since $T$ and $p$ are coprime, a contradiction. \qed

\section{Proof of Theorem \ref{thm2}: elliptic case} \label{section 5}
We now fix the notation (and assumptions) of Theorem \ref{thm2} in the elliptic setting as well as a field embedding $\overline{K}\to\overline{K_v}$. 
Let $w_0$ be the place of $L$ associated to this embedding and put $F=L_{w_0}\Q_{p^2}$. 
Recall that $E,p$ and $F$ satisfy the conditions of the situation $(S)$ and that every result of \cite[Section 3-5]{Habegger} works in this setting.
For the convenience of the reader, we state \cite[Lemma 3.3 (iii), Lemma 3.4 (ii), (iv)]{Habegger}.

\begin{lmm} \label{lmm5.-1}
Let $N\in\N$ be an integer with $p$-adic valuation $n$. 
Then: 
\begin{enumerate}
\item [(i)] $\Gal(\Q_{p^2}(p^n)/\Q_{p^2})$ acts transitively on the torsion points of order $p^n$;  
\item [(ii)] The extension $\Q_{p^2}(N)/\Q_{p^2}(N/p^n)$ is totally ramified;
\item [(iii)] If $n=1$, then $\Gal(\Q_{p^2}(N)/\Q_{p^2}(N/p))$ is cyclic of order $p^2-1$.  
\end{enumerate}
\end{lmm}

Note that $F/\Q_{p^2}$ is unramified since $v$ is unramified in $L$ by assumption.
The proof of the next lemma becomes obvious thanks to Lemma \ref{lmm5.-1} $(ii)$.

\begin{lmm} \label{lmm5.0} 
Let $N\in\N$ be an integer with $p$-adic valuation $n$. 
Then $F(N)/F(N/p^n)$ is totally ramified and $\Gal(F(N)/F(N/p^n))\simeq \Gal(\Q_{p^2}(N)/\Q_{p^2}(N/p^n))$. 
\end{lmm}

We now state our argument descent. 

\begin{lmm} \label{lmm5.25}
Let $N\in\N$ be an integer divisible by $p(p^2-1)$ with $p$-adic valuation $n$, and let $M/K$ be a totally $v$-adic finite Galois extension. 
If $\gamma\in LM(N)$ with $\sigma \gamma\in F(N/p)$ for all $\sigma\in \Gal(LM(N)/K)$, then $\gamma\in LM(N/p)$. 
\end{lmm}
\begin{proof}
By Lemma \ref{lmm4.0}, it suffices to establish that \[H:=\left\langle \psi G \psi^{-1}, \psi\in \Gal(L(N)/L)\right\rangle = \Gal(L(N)/L(N/p)), \]
 where $G=\Gal(F(N)/F(N/p))$. 
This holds when $n\geq 2$ by Lemma \ref{lmm2.8}. 
So assume that $n=1$.  
The left-hand side is the normal closure of $G\subset \Gal(L(N)/L(N/p))$ in $\Gal(L(N)/L)$; it is therefore contained in the right-hand one.
Moreover, as $p$ does not divide $N/p$, we know that $\Gal(L(N)/L(N/p))$ can identify with a subgroup of $\mathrm{GL}_2(\Z/p\Z)$. 
To obtain what we wish, it suffices to get $\# H\geq \# \mathrm{GL}_2(\Z/p\Z)$.

Let $\rho : \Gal(L(N)/L) \to \mathrm{GL}_2(\Z/p\Z)$ be the composition of the two natural maps $\Gal(L(N)/L) \to \Gal(L(p)/L)$ and $\Gal(L(p)/L)\to \mathrm{GL}_2(\Z/p\Z)$.
It is surjective by assumption and Galois theory tells us that its kernel is $\Gal(L(N)/L(p))$.
Thus, \[\rho(H)= \langle \rho(\psi) \rho(G) \rho(\psi)^{-1}, \psi\in \Gal(L(N)/L)\rangle = \langle h\rho(G) h^{-1}, h\in\mathrm{GL}_2(\Z/p\Z)\rangle.\] 
Combining Lemma \ref{lmm5.-1} $(iii)$ with Lemma \ref{lmm5.0} shows that $G$ is a cyclic group of order $p^2-1$. 
As $G \cap \Gal(L(N)/L(p)) \subset \Gal(L(N)/L(N/p))\cap \Gal(L(N)/L(p))=\{1\}$, it follows that $\rho$ restricted to $G$ is injective. 
Hence, $\rho(G)$ is a cyclic group of order $p^2-1$.
This finishes the proof since \cite[Lemma 6.1]{Habegger} gives $\rho(H)= \mathrm{GL}_2(\Z/p\Z)$.
\end{proof}

The rest of the proof faithfully follows the lines of \cite[Section 8.2]{Habegger}. 

\begin{lmm} \label{lmm5.1}
Let $N\in\N$ be an integer with $p$-adic valuation $n\geq 1$. 
We have $E(F(N)) \cap \bigcup_{m\in\N} E[p^m] = E[p^n]$. 
\end{lmm}

\begin{proof}
The inclusion $\supset$ is obvious. 
Let $T\in E(F(N))$ be a torsion point of order $p^{n'}$ and obtain $n'\leq n$. 
By Lemma \ref{lmm5.-1} $(i)$, for each $T'\in E[p^{n'}]$, there is $\sigma$ in $\Gal(\Q_{p^2}(p^{n'})/\Q_{p^2})$ such that $T'=\sigma T$.
The field $F(N)$ being Galois over $\Q_{p^2}$, we get $T'\in E(F(N))$. 
Hence, $E[p^{n'}]\subset E(F(N))$ or, equivalently, $F(p^{n'})\subset F(N)$.

The lemma is obvious if $n'=0$.
So assume that $n'\geq 1$. 
By Lemma \ref{lmm2.1} $(i)$, the extension $F(p^{n'})/F(p)$ has ramification index $p^{2(n'-1)}$. 
Next, $F(N)/F(p^n)$ is unramified by Lemma \ref{lmm2.1} $(ii)$. 
Again, Lemma \ref{lmm2.1} $(i)$ shows that the ramification index of $F(N)/F(p)$ is $p^{2(n-1)}$. 
We conclude $n'\leq n$ since $F(p^{n'})\subset F(N)$. 
\end{proof}

As $E/\Q_{p^2}$ has good reduction, the Criterion of N\'eron-Ogg-Shafarevich asserts that $F(N)/F$ is unramified for all integers $N\in\N$ coprime to $p$. 

\begin{lmm} \label{lmm5.2}
Let $N=p^nM\in\N$ be an integer with $M$ coprime to $p$ and $n\geq 1$.  
Consider $\psi\in\Gal(F(N)/F(N/p))$ and $A\in E(F(N))$ such that $B= \psi A- A\in E_\tors$. 
Then $B \in E[Q(n)]$, where $Q(n)=p^2(p^2-1)$ if $n=1$ and $Q(n)=p^2$ otherwise. 
\end{lmm}

\begin{proof}
The order of $B$ is $N'=p^{n'}M'$ for some integers $n'\geq 0$ and $M'\in\N$ coprime to $p$. 
Put $T=[p^{n'}]B$ and note that $T$ has order $M'$.

The extension $F(MM')/F$ is unramified since $MM'$ is coprime to $p$. 
As $F(M)(T)$ is included in $F(MM')$, we infer that $F(M)(T)/F(M)$ is unramified. 
Moreover, $T\in E(F(N))$, which implies that $F(M)(T)/F(M)$ is totally ramified by Lemma \ref{lmm5.0}. 
In conclusion, $T\in E(F(M))$.
In particular, $T$ is fixed by $\psi$.

The order of $[M']B\in E(F(N))$ is $p^{n'}$. 
So Lemma \ref{lmm5.1} yields $[M']B\in E[p^n]$.
Hence, $[pM']B\in E[p^{n-1}]\subset E[N/p]$ is fixed by $\psi$ too. 

B\'ezout's identity tells us that $1=ap^{n'}+bM'$ for some integers $a,b\in\Z$. 
Then $B= [a]T+ [bM']B$ and by the foregoing, we conclude that $[p]B$ is fixed by $\psi$. 
Let $t$ be the order of $\psi$. 
A small calculation proves that $B\in E[pt]$ since\[[pt]B = (\psi^{t-1}+\dots+1)([p]B)=[p](\psi^{t-1}+\dots+1)(\psi A -A)=[p](\psi^t A- A)=0.\]
Lemma \ref{lmm2.1} $(iii)$ (when $n\geq 2$) and Lemma \ref{lmm2.1} $(v)$ (if $n=1$) prove the lemma. 
\end{proof}

Recall that $\hat{h}=\hat{h}_E  : E(\overline{K}) \to \R$ denotes the Néron-Tate height.
It is non-negative, invariant under the action of $\Gal(\overline{K}/K)$ and vanishes precisely at $E_\tors$. 
It is also quadratic, that is 
\[ \forall m\in \Z, \forall P\in E(\overline{K}), \; \hat{h}([m]P)=m^2\hat{h}(P).\]
This implies 
\[\forall P\in E(\overline{K}), \forall T\in E_\tors, \; \hat{h}(P+T)=\hat{h}(P).\]
Finally, it also satisfies the parallelogram law, that is 
\[\forall P,Q\in E(\overline{K}), \hat{h}(P+Q)+\hat{h}(P-Q)=2(\hat{h}(P)+\hat{h}(Q)).\]
For more information on $\hat{h}$, we refer to \cite[Chapter VIII, §9]{Silvermanarithmetic}. 

\begin{lmm} \label{lmm5.35}
Let $P\in E(\overline{K})$, and let $\sigma\in\Gal(\overline{K}/K)$. 
If $[n]\sigma P - [m]P\in E_\tors$ for some $n,m\in\N$ distinct, then $P\in E_\tors$. 
\end{lmm}

\begin{proof}
Some properties of $\hat{h}$ recalled above show that \[m^2\hat{h}(P)=\hat{h}([m]P)=\hat{h}([m]P+[n]\sigma P-[m]P)=n^2\hat{h}(\sigma P)=n^2\hat{h}(P)  \] and the lemma follows since $n,m\in\N$ are distinct. 
\end{proof}

Let $O$ be the neutral element of $E$.
For each place $w$ of a finite extension $M$ of $K$, denote by $\lambda_w : E(\overline{M_w})\backslash \{O\} \to \R$ the local N\'eron height function on $E$ associated to $w$. 
It can be described in a totally explicit way, see \cite[Chapter VI]{Silvermanadvanced}.
For the purpose of our text, we only need to know that if $E$ has good reduction at $w$, then $\lambda_w(P)= (1/2)\max\{0; \log \vert x(P)\vert_w\}$, where $x(P)$ is the first coordinate of a point $P\in E(\overline{M_w})$ with respect to some Weierstrass model of $E/K$ that we fix from now. 

Let $A\in E(M)$. 
If $\nu$ is a place of $K$, we define the partial height function at $\nu$ as \[ \hat{h}_\nu(A)=\frac{1}{[M : \Q]} \sum_{w\vert \nu} [M_w : \Q_w] \lambda_w(A), \] where $w$ ranges over all places of $M$ above $\nu$. 
It is well known that $\hat{h}_\nu$ does not depend on the choice of the finite extension $M/K(A)$. 
By \cite[Chapter VI, Theorem 2.1]{Silvermanadvanced}, we have $\hat{h}=\sum_\nu \hat{h}_\nu$ on $E(M)$, where $\nu$ runs over all places of $M$. 
Finally, put $\hat{h}_\infty$ the sum of all $\hat{h}_\nu$, where $\nu$ runs over all infinite places of $K$. 

\begin{lmm} \label{lmm5.4}
Take an integer $N\in\N$ with $p$-adic valuation $n\geq 1$.
If $A\in E(LK^{tv}(N))$ satisfies $[Q(n)]A\notin E(F(N/p))$, then there exists a non-torsion point $B\in E(\overline{K})$ with $\hat{h}(B)\leq 4\hat{h}(A)$ and \[\hat{h}_v(B)\geq l:=\frac{\log p}{2p^4[K:\Q][F(p): \Q_p]}.\]
\end{lmm}

\begin{proof}
Let $L'\subset LK^{tv}\subset F$ be a number field, Galois over $K$, such that $A\in E(L'(N))$. 
By hypothesis, there is $\psi\in \Gal(F(N)/F(N/p))\subset \Gal(L'(N)/L'(N/p))$ such that $\psi [Q(n)]A \neq [Q(n)]A$. 
Note that $B=\psi A-A\notin E_\tors$ by Lemma \ref{lmm5.2}.
Moreover, the parallelogram law implies $\hat{h}(B)\leq 2(\hat{h}(\psi A)+\hat{h}(A))=4\hat{h}(A)$. 

Prove the lower bound for $\hat{h}_v(B)$. 
Denote by $w_0$ the place of $L'(N)$ associated to the fixed embedding $\overline{K}\to \overline{K_v}$. 
Let $C$ be the centralizer of $\psi$ in $\Gal(L'(N)/K)$. 
Let $w\in C w_0=\{\psi w_0, \psi\in C\}$. 
Then $w=\sigma^{-1}w_0$ for some $\sigma\in C$, and so \[ \vert x(B)\vert_w= \vert x(B)\vert_{\sigma^{-1}w_0} = \vert x(\sigma(\psi A - A))\vert_{w_0} = \vert x(\psi\sigma A - \sigma A)\vert_v.\]
As $\sigma B = \psi\sigma A- \sigma A \neq O$, we get $\lambda_w(B)= \lambda_v(\psi\sigma A - \sigma A)$.

Check that $\psi$ lies in the ramification group $G_s(F(N)/F)$, where $s=p^{2(n-1)}-1$.
It is obvious when $n\geq 2$ thanks to Lemma \ref{lmm2.1} $(vi)$. 
If $n=1$, then it suffices to check that $F(N)/F(N/p)$ is totally ramified, which is true thanks to Lemma \ref{lmm5.0}.

Let $\mathfrak{P}$ be the maximal ideal of the ring of integers of $F(N)$. 
Then $\psi \sigma A$ and $\sigma A$ map to same element on $E$ reduced modulo $\mathfrak{P}^{p^{2(n-1)}}$. 
Thus, $\log \vert x(\psi\sigma A - \sigma A)\vert_v \geq (p^{2(n-1)}/e)\log p$, where $e$ denotes the ramification index of $F(N)/\Q_p$. 
By Lemma \ref{lmm2.1} $(i)$, we have $e\leq p^{2(n-1)}[F(p) : \Q_p]$. 
From all of this, we get \[\lambda_w(B)\geq \frac{\log p}{2[F(p):\Q_p]}\] for all $w\in Cw_0$. 
As $L'(N)/K$ is Galois, we have $[L'(N)_w : K_v]=[L'(N)_{w_0} : K_v]$ for all places $w$ of $L'(N)$ above $v$. 
In conclusion, \[\begin{aligned}
\hat{h}_v(B) &= \frac{[K_v : \Q_p][L'(N)_{w_0} : K_v]}{[L'(N) : \Q]}\sum_{w\vert v} \lambda_w(B) \geq \frac{[L'(N)_{w_0} : \Q_p]}{[L'(N) : \Q]} \sum_{w\in Cw_0} \lambda_w(B) \\ 
& \geq \frac{[L'(N)_{w_0} : \Q_p] }{[L'(N) : \Q]} \frac{\log p}{2[F(p) : \Q_p]}\#(Cw_0) \geq  \frac{\log p}{2p^4[K:\Q][F(p): \Q_p]}
\end{aligned},\] the last inequality coming from Lemma \ref{lmm2.5} applied to $L=L'(N)$ and $H=\Gal(L'(N)/L'(N/p))$, which has cardinality at most $p^4$. 
The lemma follows.    
\end{proof}

Let $\tilde{\Phi}\in \Gal(\overline{\kappa}/\kappa)$ be the Frobenius element, where $\kappa$ denotes the residual field of $F$.
By \cite[Chapter 13, Theorem 6.3]{Husemoller}, there are $k,m\in\N$ satisfying $\tilde{\Phi}^k = [p^m]$ on $\tilde{E}$, the reduction of $E$ modulo $v$.
As $F/\Q_p$ is unramified, $\tilde{\Phi}$ identifies with an element $\Phi\in\Gal(\Q_p^{ur}/F)$.

\begin{lmm} \label{lmm5.8}
Take $N\in\N$ divisible by $p^2-1$, but not by $p^2$. 
If $A\in E(LK^{tv}(N))\backslash E_\tors$, then there is $B\in E(\overline{K})\backslash E_\tors$ with $\hat{h}(B)\leq 4p^{2m+8}\hat{h}(A)$ and $\hat{h}_v(B)\geq l$.
\end{lmm}

\begin{proof}
By replacing $N$ with $pN$ if needed, we can assume that $p\vert N$ and $p^2\nmid N$.

Let $M\subset K^{tv}$ be a number field, Galois over $K$, such that $A\in E(LM(N))$. 
Suppose that some conjugate $A'$ of $A$ over $K$ satisfies $[p^2(p^2-1)]A'\notin E(F(N/p))$. 
Then the lemma is a trivial consequence of Lemma \ref{lmm5.4} applied to $A=A'$.
So assume that $\sigma A'=[p^2(p^2-1)]\sigma A \in E(F(N/p))$ for all $\sigma\in\Gal(LM(N)/K)$, where $A'=[p^2(p^2-1)] A$. 
We can apply Lemma \ref{lmm5.25} to the coordinates of $A'$ with respect to our fixed Weierstrass model to find that $A'$ actually lies in $E(LM(N/p))$. 

The extension $F(N/p)/F$ is unramified since $p$ does not divides $N/p$.
By abuse of notation, we call $\Phi$ again the restriction of $\Phi$ to $F(N/p)$.
As $N/p$ is coprime to $p$, we have $E[N/p]\simeq \tilde{E}[N/p]$. 
As $\tilde{\Phi}^k = [p^m]$, we deduce from the last isomorphism that $\Phi^k$ acts on $E[N/p]$ as multiplication by $[p^m]$. 
By Lemma \ref{lmm4.2} applied to $N=N/p$, we conclude that $\Phi^k$ belongs to the center of $\Gal(LM(N/p)/K)$.

Put $B=\Phi^k A' -[p^m]A'$, which is non-zero by Lemma \ref{lmm5.35}.
We have \[\hat{h}(B)\leq 2(\hat{h}(\Phi^k A')+\hat{h}([p^m]A'))= 2(1+p^{2m})\hat{h}(A')\leq 4p^{2m+8}\hat{h}(A).\] 
Denote by $v_0$ the place of $LM(N/p)$ associated to the fixed embedding $\overline{K} \to \overline{K_v}$. 
Let $w$ be a place of $LM(N/p)$ above $v$.
There is $\sigma\in \Gal(LM(N/p)/K)$ such that $w=\sigma^{-1} v_0$.
A small calculation gives \[\lambda_w(B)=\frac{1}{2}\log\max\{1, \vert x(B)\vert_w\}= \frac{1}{2}\log\max\{1, \vert x(\sigma B)\vert_{v_0}\}= \lambda_v(\sigma B). \]
As $\Phi^k$ commutes with $\sigma$, we get $\sigma B= \Phi^k \sigma A' - [p^m] \sigma A'\neq O$.
However, it is clear that $\sigma B$ reduces to $O$ modulo $v$. 
Thus $\vert x(\sigma B)\vert_v \geq p$ since $F(N/p)/\Q_p$ is unramified. 
In conclusion, $\lambda_w(B)\geq (\log p)/2$ for all places $w$ of $LM(N/p)$ above $v$ and the definition of $\hat{h}_v$ leads to $\hat{h}_v(B)\geq (\log p)/2\geq l$, which finishes the proof.
\end{proof}

\begin{prop} \label{prop2}
Let $A\in E(LK^{tv}(E_\tors))\backslash E_\tors$.
Then there is $B\in E(\overline{K})\backslash E_\tors$ with $\hat{h}(B)\leq 16D^2p^{2m+8}\hat{h}(A)$ and $\hat{h}_v(B)\geq l$, where $D=2^{[F:\Q_{p^2}](p^2-1)}$.
\end{prop}

\begin{proof}
There are an integer $N\in\N$ divisible by $p^2-1$ and a number field $M\subset K^{tv}$, Galois over $K$, such that $A\in E(LM(N))$. 
Put $L'=LM$ and write $n$ for the $p$-adic valuation of $N$. 
Let $\tau=\tau_2$ be the homomorphism coming from Lemma \ref{lmm2.6} and set $C= \tau A - [D]A\in E(L'(N))$.
It is not a torsion point by Lemma \ref{lmm5.35}. 
Moreover, the parallelogram equality and other basic properties of the N\'eron-Tate height give \[\hat{h}(C)\leq 2(\hat{h}(\tau A) +\hat{h}([D]A))\leq 4D^2\hat{h}(A). \]
Let $n'\geq 0$ be the least integer such that $C\in E(L'(p^{n'-n}N))$. 
Of course, $n'\leq n$. 

If $n'\leq 1$, then Lemma \ref{lmm5.8} applied to $A=C$ provides a non-torsion point $B\in E(\overline{K})$ satisfying $\hat{h}(B)\leq 4p^{2m+8}\hat{h}(C)\leq 16 D^2p^{2m+8}\hat{h}(A)$ and $\hat{h}_v(B)\geq l$, which proves the proposition. 
So assume that $n'\geq 2$. 
By minimality of $n'$ and by Lemma \ref{lmm5.25}, there exists $\sigma\in\Gal(L'(N)/K)$ such that $C'=\sigma C \notin E(F(N'/p))$, where $N'=p^{n'-n}N$. 
Choose $\psi\in \Gal(F(N)/F(N'/p))$ such that $\psi C' \neq C'$. 

Set $A'=\sigma A$. 
As $\tau$ and $\sigma$ commute by Lemma \ref{lmm4.2}, we obtain \[ C'= \tau A' -[D]A' \in E(L'(N')).\]
To deduce the proposition, it suffices to apply Lemma \ref{lmm5.4} to $A=C'$ and $N=N'$. 
For this, we only need to show that $[p^2]C' \notin E(F(N'/p))$. 
Suppose that the contrary is true and derive a contradiction. 
Then $\psi C' - C' = T \in E[p^2]\backslash \{O\}$.
As $\psi$ and $\tau$ commute by Lemma \ref{lmm4.2}, it follows from the definition of $C'$ that \[C'+ T = \psi C' = \tau\psi A' - [D]\psi A'.  \]
A short calculation proves that $T=\tau P - [D]P$, where $P=\psi A' - A'\in E(L'(N))$. 
By Lemma \ref{lmm5.35}, $P$ is a torsion point. 
We fix $M'\in\N$ coprime to $p$ such that $[M']P$ has order a power of $p$. 
By Lemma \ref{lmm5.1}, $[M']P\in E[p^n]$ and Lemma \ref{lmm2.6} $(ii)$ ensures us that $\tau([M']P)= [DM']P$. 
Hence, $[M']T = [p^2]T=O$, which is possible only if $T=O$ since $M'$ and $p$ are coprime, a contradiction.  
\end{proof}

A proof of the next lemma can be found in \cite[Lemme 4.4]{Plessisiso}.

\begin{lmm} \label{lmm5.10}
Let $(Q_n)$ be a sequence of $E(\overline{K})\backslash E_\tors$ such that $\hat{h}(Q_n)\to 0$. 
Then $\liminf_{n\to \infty} \hat{h}_\infty(Q_n)\geq 0$. 
We also have $\liminf_{n\to \infty} \hat{h}_\nu(Q_n)\geq 0$ if $\nu$ is a finite place of $K$.
More precisely, $\hat{h}_\nu(Q_n)\geq 0$ for all $n$ if $E$ has good reduction at $\nu$.
\end{lmm}

\textit{Proof of Theorem \ref{thm2}, elliptic case:} Assume by contradiction that there exists a sequence $(A_n)$ of non-torsion points in $E(LK^{tv}(E_\tors))$ with $\hat{h}(A_n)\to 0$.
Proposition \ref{prop2} yields a new sequence $(B_n)$ of non-torsion points in $E(\overline{K})$ such that $\hat{h}(B_n)\to 0$ and $\hat{h}_v(B_n) \geq l$ for all $n$.
Lemma \ref{lmm5.10} shows that \[ \hat{h}(B_n) = \sum_\nu \hat{h}_\nu(B_n) \geq \hat{h}_v(B_n)+\hat{h}_\infty(B_n)+\sum_{\nu\in\mathcal{M}} \hat{h}_\nu(B_n) \geq l+\hat{h}_\infty(B_n)+\sum_{\nu\in\mathcal{M}} \hat{h}_\nu(B_n), \] where $\mathcal{M}$ is the (finite) set of places of $K$ with bad reduction. 
Again, Lemma \ref{lmm5.10} allows us to conclude that $\liminf_{n\to \infty} \hat{h}(B_n) \geq l$, a contradiction. \qed 

\bibliographystyle{plain}

 \end{document}